\documentclass[12pt,reqno]{amsart}
\advance \topmargin by -\headheight
\advance \topmargin by -\headsep
\evensidemargin \oddsidemargin
\usepackage{url}
\usepackage{amsmath}
\usepackage{bbm}
\usepackage{amsfonts}
\usepackage{stmaryrd}
\usepackage{amssymb}
\usepackage{amsthm}
\usepackage{csquotes}
\usepackage{color}
\usepackage{mathrsfs}
\usepackage{dsfont}
\usepackage{bm}
\usepackage{cite}
\usepackage{soul}
\usepackage{bbm}

\numberwithin{equation}{section}

\newtheorem{theorem}{Theorem}[section]
\newtheorem{corollary}{Corollary}[theorem]
\newtheorem{lemma}[theorem]{Lemma}
\newtheorem{proposition}[theorem]{Proposition}

\newtheorem{question}[theorem]{Question}
\theoremstyle{definition}
\newtheorem{definition}[theorem]{Definition}

\def\RR{\mathbb R}
\def\HH{\mathcal H}

\def\eps{\varepsilon}
\def\VC{\mathrm{VC}}

\newcommand{\ip}[2]{\left\langle #1,#2\right\rangle}
\newcommand{\norm}[1]{\left\|#1\right\|}
\newcommand{\U}{\mathsf U}

\usepackage{mathtools}

\title{A note on quantitative stability in Hilbert spaces}

\author{Yifan Jing}
\address{Department of Mathematics, The Ohio State University, Columbus OH, 43210, USA}
\email{jing.245@osu.edu}

\begin{document}

\begin{abstract}
We study stability theory in Hilbert spaces quantitatively. We prove that the inner product on the unit ball is $(k,\eps)$-stable for all $k\ge \exp(\pi/\eps)$, and it is not $(k,\eps)$-stable for $k\le \exp(\log 2/\eps)$, showing that the growth is necessarily exponential in $1/\eps$.

We then analyze how stability scales under nonlinear connectives applied to the inner product. In particular, for power-type predicates $f(x,y)=\langle x,y\rangle_+^\beta$ with $\beta<1$ we obtain upper and lower bounds of the form $\exp(C\eps^{-1/\beta})$, and for $\beta>1$ and integer powers $\langle x,y\rangle^d$ we retain the bilinear scale $\exp(C/\eps)$. 
\end{abstract}

\maketitle

\section{Introduction}

The stability of the Hilbert inner product has a long history before its appearance in continuous logic. On the functional analytic side, Grothendieck~\cite{Grothendieck} introduced a ``double limit'' paradigm in which commutation of iterated limits serves as a compactness criterion. The connection between the compactness criterion and the stability theory in model theory was made explicit by Ben Yaacov \cite{BenStab}, although related ideas appear to have been used implicitly much earlier. For related discussions see also~\cite{Pillay_Groth,Conant_Stab}.
 In a different direction, Krivine and Maurey~\cite{KrivineMaurey} formulated a notion of stable Banach space by requiring suitable double limits of norms along type-determining sequences to commute. 
These classical tools yield qualitative stability statements for the inner product in Hilbert spaces: for each $\eps>0$ there exists some $k(\eps)$ such that no $(k(\eps),\eps)$-half-graph appears. On the quantitative side, $\eps$ double-limit defect implies the set is $\delta$-weakly relatively compact, with explicit dependence between $\eps$ and $\delta$ \cite{FHMZ}, and it was further developed and extended by many authors, see e.g. \cite{Granero, QuantitativeBanach} and the references therein.

In this note we work with the following quantitative half-graph notion.

\begin{definition}[$(k,\eps)$-half-graphs and stability]
\label{def:stability}
Let $f:X\times Y\to\RR$.  We say that $f$ has a \emph{$(k,\eps)$-half-graph} if there exist
$x_1,\dots,x_k\in X$ and $y_1,\dots,y_k\in Y$ such that for all $i<j$,
\begin{equation}\label{eq:intro-halfgraph}
f(x_i,y_j)-f(x_j,y_i) \ge \eps.
\end{equation}
If no such configuration exists, we say that $f$ is \emph{$(k,\eps)$-stable}.
\end{definition}

In the literature, $(k,\varepsilon)$-stability is sometimes defined using an absolute value on the left-hand side of \eqref{def:stability}; see, for instance, \cite{CCP24}. Qualitatively, the two definitions are equivalent, and quantitatively they differ by at most an exponential factor. We believe that the definition adopted here directly encodes the edge/non-edge structure of half graphs. The same definition also appears in~\cite{Talagrand87, Talagrand96}.

In continuous logic, the stability of a binary formula can be formulated in terms of the absence of half-graphs as in Definition~\ref{def:stability}. The modern framework for continuous first order logic was developed by Ben Yaacov, Berenstein, Henson, and Usvyatsov~\cite{Continuous_Logic}, and further developed by many authors. In this setting the theory of real or complex Hilbert spaces is stable, and definable predicates admit robust geometric descriptions; in particular, the inner product on the unit ball is stable. What is typically not tracked is the effective dependence of $k$ on the error tolerance $\eps$, even for the basic predicate $\langle x,y\rangle$.

On the application side, the stability theory of Hilbert space  plays a role in the model theory of stable groups and approximate groups by Hrushovski~\cite{Hrushovski}, and it is a key ingredient in the simplified proof of Tao's algebraic regularity lemma~\cite{Tao} due to Pillay and Starchenko~\cite{Pillay_Starchenko}, see also \cite{TaoBlog}.

We now state our first result. Let $\HH$ be a Hilbert space and let $\U=\{x\in \HH:\|x\|\le 1\}$ be its unit ball.
Our first main theorem gives an explicit and essentially optimal quantitative stability bound.

\begin{theorem}[Quantitative stability of the Hilbert inner product]
\label{thm:inner product}
Let $\HH$ be a Hilbert space, $\eps\in(0,1)$, and $\U\subseteq \HH$ its unit ball.
Then the inner product $\langle\cdot,\cdot\rangle:\U\times\U\to[-1,1]$ is $(k,\eps)$-stable for
\[
k\ge \exp\Big(\frac{\pi}{\eps}\Big).
\]
Moreover, for $m=  1/\eps $ there exist vectors $x_1,\dots,x_{2^m},y_1,\dots,y_{2^m}\in\U$ such that
$\langle x_i,y_j\rangle-\langle x_j,y_i\rangle\ge 1/m\ge \eps$ for all $i<j$.
In particular, $\langle\cdot,\cdot\rangle$ is not $(k,\eps)$-stable for
\[
k \le\exp\Big(\frac{\log 2}{\eps}\Big).
\]
\end{theorem}

This answers a question asked by Conant at the \emph{Combinatorics Meets Model Theory} workshop in Cambridge; see also~\cite{CCP24} for related work\footnote{The question asked in the paper is for a slightly  different notion of stability, namely $|f(x_i,y_j)-f(x_j,y_i)|>\eps$. Our lower  bound construction still gives $\exp(C/\eps)$ while the upper bound will be $\exp\exp(C/\eps)$}. Our results also raise a natural sharp-constant problem in the spirit of classical functional-analytic extremal inequalities\footnote{The question was recently answered by the author with the help of ChatGPT. This result is included in the appendix. The sharp bound is $\exp(\pi/\eps)$. }. 
\begin{question}[See Appendix B]\label{question}
   Determine the optimal constant $C_i$, where $C_i$ is the infimum over all $C>0$ with the property that, for every Hilbert space $\HH$ and every $\eps\in(0,1)$, the inner product on the unit ball is $(k,\eps)$-stable for all $k>\exp(C/\eps)$. 
\end{question}

A second theme of this paper concerns how quantitative stability behaves under nonlinear operations on the inner product, and, more generally, how the regularity of a one-variable connective influences the size of half-graphs. From a purely analytic viewpoint, it is natural to expect that passing from $\langle x,y\rangle$ to a nonlinear expression such as $(\langle x,y\rangle)^2$, or to predicates involving norms (e.g. $\|x-y\|$), might increase the effective ``combinatorial complexity'' of half-graphs and lead to bounds of the form $\exp(C/\eps^2)$ or worse. What happens is subtler.
For instance, using the tensor-power trick, integer-power connectives $t\mapsto t^d$ can be treated on the same qualitative scale as Theorem~\ref{thm:inner product}. On the other hand, for power-type connectives with $\beta\in(0,1]$, the correct scale is $\exp(C\eps^{-1/\beta})$. More generally, for monotone connectives with a quantitative H\"older modulus on the relevant range, one obtains bounds of the form $\exp(C/\eps^\alpha)$, where the exponent $\alpha$ reflects the effective ``flatness'' of the connective at scale $\eps$.
We develop this systematically in the following theorem.

\begin{theorem}[Quantitative stability for power connectives]
\label{thm:holder conn}
There exist absolute constants $C,c>0$ such that the following hold.
\begin{enumerate}
\item Fix $\alpha\in[1,\infty)$ and define
\[
f_\alpha(x,y) := \big(\langle x,y\rangle_+\big)^\alpha,
\qquad \text{where } t_+:=\max\{t,0\}.
\]
Then $f_\alpha$ on $\U\times\U$ is $(k,\eps)$-stable for all $k\ge\exp(C_\alpha/\eps)$, and it is not $(k,\eps)$-stable for
$k\le \exp(c_\alpha/\eps)$, where $c_\alpha,C_\alpha>0$ depend only on $\alpha$.

\item Fix $\beta\in(0,1]$ and define
\[
f_\beta(x,y) := \big(\langle x,y\rangle_+\big)^\beta.
\]
Then $f_\beta$ on $\U\times\U$ is $(k,\eps)$-stable for all $k\ge\exp(C/\eps^{1/\beta})$, and it is not $(k,\eps)$-stable for
$k\le \exp(c/\eps^{1/\beta})$.

\item Fix an integer $d\ge 1$ and define
\[
f_d(x,y) := \langle x,y\rangle^d.
\]
Then $f_d$ on $\U\times\U$ is $(k,\eps)$-stable for all $k\ge\exp(C_d/\eps)$, and it is not $(k,\eps)$-stable for
$k\le \exp(c_d/\eps)$, where $c_d,C_d>0$ depend only on $d$.
\end{enumerate}
\end{theorem}

\subsection*{Organization}
 Sections~\ref{sec:lower} and \ref{sec:upper} prove Theorem~\ref{thm:inner product}: we give an explicit half-graph construction for the inner product and a matching exponential upper bound via Grothendieck's inequality type argument. Section~\ref{sec:powers} studies nonlinear connectives of the inner product, including power-type predicates, and establishes sharp quantitative bounds of the form $\exp(C/\eps^\alpha)$. 
In the appendix, we discuss related notions of dimension, including continuous analogues of VC-dimension. We determine the exact dimensions for the inner product in this setting.

\subsection*{Notation} we fix notation and recall the continuous logic used in the later part of the paper. Throughout $\eps\in(0,1)$ is a parameter. Let $\HH$ be a real or complex Hilbert space with inner product $\langle\cdot,\cdot\rangle$ and the $L^2$-norm $\|x\|=\langle x,x\rangle^{\frac12}$.
We write
\[
\U:=\{x\in \HH:\|x\|\le 1\}
\]
for the unit ball.  We denote by $B(\HH)$ the algebra of all bounded linear operators $T:\HH\to \HH$, equipped with the operator norm
\[
\|T\|_{\mathrm{op}}=\sup_{\|x\|\le 1}\|Tx\|.
\]
In the complex case all predicates are understood to be real-valued, thus the inner product is $\Re\langle x,y\rangle$. The upper bound arguments below apply equally well to complex Hilbert spaces after passing to the underlying real Hilbert space. The lower bound  constructions are written in real Hilbert spaces, and therefore also give lower bounds in complex Hilbert spaces by embedding the real examples as real subspaces.

We work in the standard setting of continuous first-order logic, but we will only need concrete languages on the unit ball $\U$.
Let $X$ be a set,  A predicate on $X$ is a bounded real-valued function $
\varphi:X\to \RR.
$
Let $\mathcal L_{\mathrm{Hilb}}$ be the single-sorted language on $\U$ containing:
\begin{enumerate}
\item a constant symbol $0$;
\item binary function symbols $f_{\lambda,\mu}(x,y)=\lambda x+\mu y$ for each $\lambda,\mu\in\mathbb Q$ with $|\lambda|+|\mu|\le 1$;
\item the metric predicate $d(x,y)=\|x-y\|$;
\item the inner product predicate $\langle x,y\rangle$ (or $\Re\langle x,y\rangle$ in the complex case).
\end{enumerate}
All symbols are uniformly continuous on $\U$.

In the proof, for the sake of a cleaner presentation, we will ignore floor and ceiling functions when integer parameters are required. This introduces at most an $O(1)$ error.

\section*{Statement of AI}

The first version of the paper was done by the author, and ChatGPT was used for polishing the language. In the new version, Appendix B is added and the new construction in the appendix was found by ChatGPT. Appendix B is written by the author and he is responsible for the mathematical content. 

 \section*{Acknowledgements}
 The author is grateful to Gabriel Conant for the wonderful lectures he gave at the \emph{Combinatorics Meets Model Theory} workshop in Cambridge, and to the other organizers, Caroline Terry and Julia Wolf. He also thanks Jinhe Ye for bringing \cite{BenStab} to his attention, Artem Chernikov for valuable discussion regarding stability and the VC-dimension, 
  and Chavdar Lalov for introducing him to basic concepts from learning theory.
  The author is supported by NSF Grant DMS-2503063.

\section{A half-graph construction in a Hilbert space}\label{sec:lower}

Let $\HH$ be a Hilbert space and let $\U:=\{x\in \HH:\norm{x}\le 1\}$ be its unit ball. Recall that a function $f:X\times Y\to\mathbb R$ is $(k,\eps)$-stable if there do not exist $x_1,\dots,x_k\in X$ and $y_1,\dots,y_k\in Y$ such that for every $i<j$,
\[
f(x_i,y_j)-f(x_j,y_i)\ge \eps.
\]
In this section we show that the exponential upper bound in $1/\eps$ is sharp (up to constants in the exponent) already for the inner product on $\U$.

Fix an integer $m\ge 1$ and let $\Sigma_m:=\{0,1\}^m$ be the set of binary strings of length $m$. Write $s=s_1\cdots s_m\in\Sigma_m$, and for $0\le r\le m$ denote by $s_{\le r}:=s_1\cdots s_r$ the length-$r$ prefix, with $s_{\le 0}=\varnothing$. We order $\Sigma_m$ lexicographically, that is $s\prec t$ if at the first index where they differ one has $s_r=0$ and $t_r=1$.

Consider the full binary tree of depth $m$. Its vertices are the prefixes
\[
V=\{\varnothing\}\cup\{0,1\}\cup\cdots\cup\{0,1\}^{m},
\]
and its directed edges are of the form $(v\to v0)$ and $(v\to v1)$ where $v\in\{0,1\}^{<m}$. Let $\mathsf E$ denote the set of directed edges. Define the Hilbert space
\[
\HH:=\ell_2(\mathsf E)
\]
with orthonormal basis $\{e_{(v\to vb)}:(v\to vb)\in\mathsf E\}$.

For each $s\in\Sigma_m$, define vectors $x_s,y_s\in \HH$ by
\begin{align}
x_s &:=  m^{-\frac12}\sum_{r=1}^m e_{(s_{\le r-1}\to s_{\le r})}, \label{eq:def-xs}\\
% y_s &:= \frac{1}{\sqrt m}\sum_{r=1}^m e_{(s_{\le r-1}\to s_{\le r-1}0)}. \label{eq:def-ys}
y_s &:= m^{-\frac12}\sum_{\substack{1\le r\le m\\ s_r=1}} e_{(s_{\le r-1}\to s_{\le r-1}0)}. \label{eq:def-ys}
\end{align}
We have the following easy observation. 

\begin{lemma}
\label{lem:unit}
For every $s\in\Sigma_m$ one has $x_s,y_s\in\U$.
\end{lemma}

\begin{proof}
Each of $x_s$ is the average of $m$ distinct basis vectors $e_e$ with orthonormal supports, so its norm is $1$. Each of $y_s$ is the average of at most $m$ distinct basis vectors $e_e$ with orthonormal supports, so its squared norm is at most $1$. 
\end{proof}

For $s,t\in\Sigma_m$, note that
\begin{equation}\label{eq:inner-as-overlap}
\ip{x_s}{y_t}=\frac{1}{m}\Big|\Big\{1\le r\le m: (s_{\le r-1}\to s_{\le r})=(t_{\le r-1}\to t_{\le r-1}0),e_{(t_{\le r-1}\to t_{\le r-1}1)}\in t\Big\}\Big|.
\end{equation}
In words, $\ip{x_s}{y_t}$ is $1/m$ times the number of levels $r$ at which the $x$-edge of $s$ matches the left edge at the corresponding prefix of $t$, whenever $e_{(t_{\le r-1}\to t_{\le r-1}1)}$ is an edge in $t$. 

\begin{proposition}[Half-graph witness]
\label{prop:key}
If $s,t\in\Sigma_m$ satisfy $s \prec t$ in lexicographic order, then
\[
\ip{x_s}{y_t}-\ip{x_t}{y_s}=\frac{1}{m}.
\]
In particular, for all $s\prec t$ one has $\ip{x_s}{y_t}-\ip{x_t}{y_s}\ge 1/m$.
\end{proposition}

\begin{proof}
    Fix $s\prec t$ and let $r$ be the first index where they differ. Then $s_r=0$, $t_r=1$, and suppose $s_{\le r-1}=t_{\le r-1}=p$. We compare the overlap counts in \eqref{eq:inner-as-overlap}.

\medskip\noindent
\emph{Claim 1: $\langle x_s,y_t\rangle=1/m$.}\medskip

At level $r$, the $x_s$-edge is $(p\to p0)$. Since $t_r=1$, the sum defining $y_t$ includes at level $r$ the left edge $(p\to p0)$, contributing $1/m$. For levels $\ell<r$, if $t_\ell=1$ then $y_t$ uses the left edge at that prefix while $x_s$ uses the right edge at that prefix since $s_\ell=t_\ell=1$, so there is no match. If $t_\ell=0$ then $y_t$ contributes nothing. For levels $\ell>r$, the prefixes differ so the edges are distinct. Hence there is exactly one match, at level $r$, giving $\langle x_s,y_t\rangle=1/m$.

\medskip\noindent
\emph{Claim 2: $\langle x_t,y_s\rangle=0$.}\medskip

At level $r$, since $s_r=0$, the sum defining $y_s$ includes nothing. For $\ell<r$, if $s_\ell=1$ then $y_s$ uses the left edge while $x_t$ uses the right edge at that prefix since $t_\ell=s_\ell=1$, so there is no match. If $s_\ell=0$ then $y_s$ contributes nothing. For $\ell>r$ the prefixes differ. Therefore there are no matches, and $\langle x_t,y_s\rangle=0$.\medskip

The two claims together imply the desired conclusion. 
\end{proof}

We now prove the main result of the section. 

\begin{theorem}
\label{thm:lower}
Let $\HH$ be a Hilbert space of dimension at least $2^{m+1}-2$ and let $\U$ be its unit ball.
Then the inner product $\ip{\cdot}{\cdot}:\U\times\U\to[-1,1]$ is not $(2^m,1/m)$-stable.
In particular, the inner product over unit ball is not $(k,\eps)$-stable for 
\[
k \le 2^{  1/\eps }
 = \exp\Big(\frac{\log 2}{\eps}\Big). 
\]
\end{theorem}

\begin{proof}
Let $m=  1/\eps $ and construct $\HH=\ell_2(\mathsf E)$ and vectors $\{x_s,y_s\}_{s\in\Sigma_m}\subseteq \HH$ as in \eqref{eq:def-xs} and \eqref{eq:def-ys}. By Lemma~\ref{lem:unit}, all these vectors lie in $\U$.

Enumerate $\Sigma_m$ in increasing lexicographic order as $s^{(1)}\prec \cdots\prec s^{(2^m)}$ and set
\[
x_i:=x_{s^{(i)}},\qquad y_i:=y_{s^{(i)}}.
\]
for $1\le i\le 2^m$. 
Then for every $i<j$, Proposition~\ref{prop:key} gives
\[
\ip{x_i}{y_j}-\ip{x_j}{y_i}=\frac{1}{m}\ge \eps.
\]
Thus the family $\{x_i\}_{i=1}^{2^m}$ and $\{y_i\}_{i=1}^{2^m}$ witnesses that the inner product fails to be $(2^m,\eps)$-stable.
\end{proof}

\section{An exponential upper bound via the discrete Hilbert transform}\label{sec:upper}

The following theorem is the upper bound part of Theorem~\ref{thm:inner product}. 

\begin{theorem}[Upper bound for the inner product]
\label{thm:upper}
Let $\HH$ be a Hilbert space, $\eps\in(0,1)$, and $\U\subseteq \HH$ be the unit ball. Then the inner product $\ip{\cdot}{\cdot}:\U\times \U\to[-1,1]$ is $(k,\eps)$-stable for
\[
k \ge \exp\Big(\frac{\pi}{\eps}\Big).
\]
\end{theorem}

It is well-known that bilinear forms in Hilbert spaces is controlled by Grothendieck's inequality \cite{GrothendieckIneq} up to universal constants. In particular, for a matrix $A=(a_{ij})$, define the cut norm
\begin{equation}\label{eq: cut}
\| A\|_{\pm1} = \sup _{\eps_i,\delta_j=\pm1} \sum a_{ij}\eps_i\delta_j,
\end{equation}
then Grothendieck's inequality asserts that for any $\|x_i\|,\|y_j\|\le 1$, 
\[
\sum a_{ij} \langle x_i,y_j\rangle \leq K_G \|A\|_{\pm1}. 
\]
Here $K_G$ a universal constant. Determining exact value of $K_G$ is a well-known open problem in the field, notably Krivine's rounding scheme \cite{Krivine_rounding} provide the state of art upper bound, Braverman,  Makarychev, Makarychev, and Naor \cite{BMMN} show that the exact value of $K_G$ should be strictly smaller than Krivine's bound.

Now we assume $A$ is a real skew-symmetric matrix, that is $a_{ij}=-a_{ji}$, and with diagonal entries all equal to $0$. Given $\eps>0$, assume $x_1,\dots,x_k,y_1,\dots,y_k$ is a $k$-half-graph witnessed by the inner product in a Hilbert space $\HH$. Hence whenever $a_{ij}\geq0$ when $i<j$
\[
\sum a_{ij}\langle x_i,y_j\rangle=\sum_{i<j} a_{ij} \big(\langle x_i,y_j\rangle - \langle x_j,y_i\rangle\big)\geq \eps \sum_{i<j} a_{ij}. 
\]
On the other hand, by Grothendieck's inequality,
\[
\sum a_{ij}\langle x_i,y_j\rangle \leq K_G \|A\|_{\pm 1}, 
\]
hence
\[
\eps \leq \frac{K_G \|A\|_{\pm1}}{\sum_{i<j}a_{ij}},
\]
hence the whole problem reduces to designing a skew-symmetric matrix $A$ with large $\sum_{i<j}a_{ij}$ and small $\|A\|_{\pm1}$.

In our proof we make use of the discrete Hilbert transform. Hence our matrix is explicit, we do not need to apply Grothendieck-type inequalities in our proof. Nevertheless, the author believes the above argument reveals the analytic picture behind the proof. The idea of introducing the Hilbert transform in this context goes back to Tao~\cite{TaoBlog}, in his simplified proof of the algebraic regularity lemma. 

\begin{lemma}[Discrete Hilbert transform]
\label{lem:hilbert}
Fix $n\ge 2$. Let $(u_i)_{i=1}^n\subset \mathbb C$.
Then
\[
\left|\sum_{1\le i\ne j\le n}\frac{1}{i-j}u_i \overline{u_j}\right|
\;\le\;
\pi \sum_{i=1}^n |u_i|^2.
\]
Equivalently, the $n\times n$ matrix $A=(a_{ij})$ with $a_{ii}=0$ and
$a_{ij}=1/(i-j)$ for $i\ne j$ satisfies $\|A\|_{\mathrm{op}}\le \pi$.
\end{lemma}

We will also use the following standard vector-valued extension: matrices bounded on $\ell_2$
remain bounded when acting on $\ell_2$-sequences with values in a Hilbert space.

\begin{lemma}[Vector-valued $\ell_2$ bound]
\label{lem:vector}
Let $A=(a_{ij})_{i,j=1}^n$ be a complex matrix with $\|A\|_{\mathrm{op}}\le M$.
Then for any Hilbert space $\HH$ and any vectors $y_1,\dots,y_n\in \HH$,
\[
\sum_{i=1}^n \left\|\sum_{j=1}^n a_{ij}y_j\right\|^2 \le M^2 \sum_{j=1}^n \|y_j\|^2.
\]
\end{lemma}

\begin{proof}
Let $(e_\ell)_{\ell\in\Lambda}$ be an orthonormal basis of the closed span of $\{y_j\}$.
Write $y_j=\sum_\ell y_{j,\ell} e_\ell$ with $y_{j,\ell}\in\mathbb C$. Then
\[
\sum_{i=1}^n \left\|\sum_{j=1}^n a_{ij}y_j\right\|^2
=
\sum_{i=1}^n \sum_{\ell}\left|\sum_{j=1}^n a_{ij}y_{j,\ell}\right|^2
=
\sum_{\ell}\|A(y_{1,\ell},\dots,y_{n,\ell})\|_2^2
\le
M^2\sum_{\ell}\sum_{j=1}^n |y_{j,\ell}|^2,
\]
and the right-hand side equals $M^2\sum_{j=1}^n\|y_j\|^2$ by Parseval.
\end{proof}

\begin{proof}[Proof of Theorem~\ref{thm:upper}]
Assume for contradiction that there exist $x_1,\dots,x_n\in \U$ and $y_1,\dots,y_n\in\U$
such that for all $i>j$,
\begin{equation}\label{eq:halfgraph}
\ip{x_i}{y_j}-\ip{x_j}{y_i}\ge \eps.
\end{equation}
Let $A=(a_{ij})_{i,j=1}^n$ be the $n\times n$ matrix with $a_{ii}=0$ and when $i\neq j$
\[
a_{ij}=\frac{1}{i-j}.
\]
Note that $A$ is real skew-symmetric and by Lemma~\ref{lem:hilbert}, $\|A\|_{\mathrm{op}}\le \pi$.

Consider the scalar
\[
S:=\sum_{1\le i\ne j\le n} a_{ij}\ip{x_i}{y_j}.
\]
By linearity of the inner product,
\[
S=\sum_{i=1}^n \ip{x_i}{\sum_{j\ne i} a_{ij}y_j}.
\]
Applying Cauchy--Schwarz and Lemma~\ref{lem:vector}  we obtain
\begin{align*}
|S|
&\le \left(\sum_{i=1}^n \|x_i\|^2\right)^{1/2}
\left(\sum_{i=1}^n \left\|\sum_{j\ne i} a_{ij}y_j\right\|^2\right)^{1/2}
\\
&\le \left(\sum_{i=1}^n \|x_i\|^2\right)^{1/2}
 \|A\|_{\mathrm{op}} \left(\sum_{j=1}^n \|y_j\|^2\right)^{1/2}
\le \pi n,
\end{align*}
since $\|x_i\|\le 1$ and $\|y_j\|\le 1$.

On the other hand, using $a_{ji}=-a_{ij}$ and grouping terms in pairs $(i,j)$ with $i<j$,
\begin{align*}
S
&=\sum_{1\le i<j\le n} a_{ij}\ip{x_i}{y_j} + a_{ji}\ip{x_j}{y_i}
=\sum_{1\le i<j\le n}\frac{1}{i-j}\big(\ip{x_i}{y_j}-\ip{x_j}{y_i}\big)
\\
&=\sum_{1\le i<j\le n}\frac{1}{j-i}\big(\ip{x_j}{y_i}-\ip{x_i}{y_j}\big).
\end{align*}
By \eqref{eq:halfgraph}, each bracketed term is at least $\eps$, hence
\[
S\ge \eps\sum_{1\le i<j\le n}\frac{1}{j-i}
\ge \eps(n\log n-n).
\]

Combining the two bounds gives
\[
\eps n\log n \le S \le |S|\le \pi n,
\]
so $\log n \le \pi/\eps$, i.e. $n\le \exp(\pi/\eps)$.
Therefore no such configuration \eqref{eq:halfgraph} can exist for $n>\exp(\pi/\eps)$, which exactly means that $\ip{\cdot}{\cdot}$ is $(k,\eps)$-stable for $k\ge   \exp(\pi/\eps) $.
\end{proof}

%=================================================================
\section{Power and H\"older connectives of the Hilbert inner product}
\label{sec:powers}

In this section we record quantitative stability bounds for predicates obtained by composing the Hilbert
inner product with power and H\"older type connectives. 
Throughout, $\HH$ is a real Hilbert space, $\U$ its unit ball, and $\langle\cdot,\cdot\rangle$
the inner product.  The complex case is identical by replacing $\langle\cdot,\cdot\rangle$ by $\Re\langle\cdot,\cdot\rangle$.

In general, as the binary predicates are no longer bilinear, Grothendieck's inequality type idea used in the upper bound for the inner product case stop working, and the lower bound constructions for the inner product case are usually not sharp. Hence we are going to treat connectives separately via different methods.

Let us first look at the case when $\beta\in(0,1]$ is fixed, and define
\[
f_\beta(x,y):=\big(\langle x,y\rangle\big)_+^\beta,
\qquad\text{where }t_+:=\max\{t,0\}.
\]
The lower bound on stability can be obtained directly from our localized half graph construction. 

\begin{corollary}[Sharp lower bounds for $t^\beta$, $0<\beta\le 1$]
\label{cor:lower beta}
For every integer $m\ge 1$, the predicate $f_\beta$ on $\U\times\U$ is not $(2^m,m^{-\beta})$-stable.
Equivalently, for every $\eps\in(0,1)$, letting $m:=  \eps^{-1/\beta} $ gives a $(k,\eps)$-half-graph with
\[
k=2^m \ge \exp\big(c\eps^{-1/\beta}\big),
\]
where $c>0$ is an absolute constant. 
\end{corollary}

\begin{proof}
Apply Proposition~\ref{prop:key}. For $s<t$ we have $\langle x_s,y_t\rangle=1/m$ and $\langle x_t,y_s\rangle=0$,
hence
\[
f_\beta(x_s,y_t)-f_\beta(x_t,y_s)=\Big(\frac{1}{m}\Big)^\beta-0=m^{-\beta}.
\]
Enumerate $\Sigma_m$ in lexicographic order to obtain a $(2^m,m^{-\beta})$-half-graph.
Taking $m=  \eps^{-1/\beta} $ gives $m^{-\beta}\ge \eps$ and $k=2^m\ge \exp(c\eps^{-1/\beta})$.
\end{proof}

We now show that the lower bound exponent $\eps^{-1/\beta}$ is also the correct order for the upper bound.
This is a direct consequence of the $\beta$-H\"older modulus of $t\mapsto t^\beta$ on $[0,1]$ together with the
$\exp(C/\eta)$ upper bound for the inner product (Theorem~\ref{thm:upper}).

\begin{lemma}
\label{lem:holder}
Fix $\beta\in(0,1]$. Then for all $s,t\in[0,1]$,
\[
|t^\beta-s^\beta|\le |t-s|^\beta.
\]
Equivalently, if $t^\beta-s^\beta\ge \eps$ then $t-s\ge \eps^{1/\beta}$.
\end{lemma}

\begin{proof}
The function $u\mapsto u^\beta$ is concave and increasing on $[0,1]$, and one has
$(s+\Delta)^\beta-s^\beta\le \Delta^\beta$ for $\Delta\ge 0$, which yields the stated inequality.
The contrapositive gives the second sentence.
\end{proof}

\begin{corollary}[Upper bound for $t^\beta$]
\label{cor:upper beta}
Fix $\beta\in(0,1]$ and let $f_\beta(x,y)=(\langle x,y\rangle)_+^\beta$ on $\U\times\U$.
Then there exists a constant $C_\beta<\infty$ such that for every $\eps\in(0,1)$ the predicate $f_\beta$ is $(k,\eps)$-stable for all
\[
k>\exp\big(C\eps^{-1/\beta}\big).
\]
More concretely, one may take $C=\pi$.
\end{corollary}

\begin{proof}
Assume $f_\beta$ has a $(k,\eps)$-half-graph, witnessed by $x_1,\dots,x_k,y_1,\dots,y_k\in\U$ such that for all $i<j$,
\[
(\langle x_i,y_j\rangle)_+^\beta-(\langle x_j,y_i\rangle)_+^\beta\ge \eps.
\]
Since $t\mapsto (t)_+^\beta$ is increasing on $\RR$, Lemma~\ref{lem:holder} implies, for $i<j$,
\[
\langle x_i,y_j\rangle-\langle x_j,y_i\rangle  \ge\eps^{1/\beta}.
\]
Thus the inner product itself has a $(k,\eps^{1/\beta})$-half-graph.  Applying Theorem~\ref{thm:upper} at scale
$\eta=\eps^{1/\beta}$ yields $k\le \exp(C/\eta)=\exp(C\eps^{-1/\beta})$, as claimed.
\end{proof}

We next look at the integer power case.  In contrast to the H\"older case,
integer powers do \emph{not} worsen the exponential rate in $\eps$ on the upper-bound side. Here we use a tensor-power trick. 

\begin{lemma}[Tensor-power linearization]
\label{lem:tensor}
Fix $d\in\mathbb N$.  For any Hilbert space $\HH$ and any $x,y\in \HH$ one has
\[
\langle x,y\rangle^d=\big\langle x^{\otimes d},y^{\otimes d}\big\rangle_{\HH^{\otimes d}}.
\]
Moreover, if $\|x\|\le 1$ then $\|x^{\otimes d}\|\le 1$.
\end{lemma}

\begin{proof}
Choose an orthonormal basis and expand $x=\sum_i x_i e_i$, $y=\sum_i y_i e_i$.
Then $x^{\otimes d}=\sum_{i_1,\dots,i_d} x_{i_1}\cdots x_{i_d}e_{i_1}\otimes\cdots\otimes e_{i_d}$, and similarly for $y^{\otimes d}$.
Taking the inner product in $\HH^{\otimes d}$ gives
\[
\langle x^{\otimes d},y^{\otimes d}\rangle=\sum_{i_1,\dots,i_d} x_{i_1}\cdots x_{i_d}\overline{y_{i_1}\cdots y_{i_d}}
=\Big(\sum_i x_i\overline{y_i}\Big)^d=\langle x,y\rangle^d.
\]
Also $\|x^{\otimes d}\|=\|x\|^d$.
\end{proof}
The tensor-power trick would immediately give us the following upper bound. 

\begin{corollary}[Upper bound for $\langle x,y\rangle^d$]
\label{cor:upper integer power}
Fix $d\in\mathbb N$ and define $f_d(x,y):=\langle x,y\rangle^d$ on $\U\times\U$.
Then there exists $C<\infty$ such that $f_d$ is $(k,\eps)$-stable for all $k>\exp(C/\eps)$. More concretely, one may take $C=\pi$.
\end{corollary}

\begin{proof}
If $f_d$ has a $(k,\eps)$-half-graph on $\U\times\U$, then by Lemma~\ref{lem:tensor}
the inner product on the unit ball of $\HH^{\otimes d}$ has a $(k,\eps)$-half-graph witnessed by $x_i^{\otimes d},y_i^{\otimes d}$.
Apply Theorem~\ref{thm:upper} in $\HH^{\otimes d}$ to conclude $k\le \exp(C/\eps)$.
\end{proof}

Note that for $f_d$, a native application of Proposition~\ref{prop:key} would only provide a $(k,\eps)$ lower bound with $k\geq \exp(C/\eps^{1/d})$.  To obtain a sharp bound, we will use a shift of basis idea to overcome the ``flatness'' issue as scale $\eps$, and we will combine the integer power case together with the remained case, when the connective $t^\alpha$ satisfies $\alpha\in[1,\infty)$.

In the regime $\alpha>1$, the power map $t\mapsto t^\alpha$ is globally Lipschitz on $[0,1]$
but ``flat'' at $0$.

\begin{lemma}
\label{lem:lipschitz for alpha}
Fix $\alpha\ge 1$.  Then $t\mapsto t^\alpha$ is $\alpha$-Lipschitz on $[0,1]$: for $s,t\in[0,1]$
\[
|t^\alpha-s^\alpha|\le \alpha|t-s|.
\]
\end{lemma}

\begin{proof}
By the mean value theorem,
$|t^\alpha-s^\alpha|=\alpha\xi^{\alpha-1}|t-s|\le \alpha|t-s|$ for some $\xi\in[s,t]\subset[0,1]$.
\end{proof}

\begin{proposition}[Upper bound for $(\langle x,y\rangle)_+^\alpha$ with $\alpha\ge 1$]
\label{prop:upper alpha}
Fix $\alpha\ge 1$ and define
\[
f_\alpha(x,y):=\big(\langle x,y\rangle\big)_+^\alpha. 
\]
Then there exists $C_\alpha<\infty$ such that for every $\eps\in(0,1)$ the predicate $f_\alpha$ on $\U\times \U$ is $(k,\eps)$-stable for all
\[
k>\exp\big(C_\alpha/\eps\big).
\]
More concretely, one may take $C_\alpha=\alpha\pi$.
\end{proposition}

\begin{proof}
Assume $f_\alpha$ has a $(k,\eps)$-half-graph witnessed by $x_1,\dots,x_k,y_1,\dots,y_k\in\U$, so for all $i<j$,
\[
(\langle x_i,y_j\rangle)_+^\alpha-(\langle x_j,y_i\rangle)_+^\alpha \ge \eps.
\]
Since $t\mapsto (t)_+^\alpha$ is increasing and $\alpha$-Lipschitz on $[-1,1]$ (by Lemma~\ref{lem:lipschitz for alpha} on $[0,1]$),
we have
\[
\langle x_i,y_j\rangle-\langle x_j,y_i\rangle \ge \frac{\eps}{\alpha}
\]
whenever $i<j$. 
Thus the inner product itself has a $(k,\eps/\alpha)$-half-graph.  Applying Theorem~\ref{thm:upper} at scale
$\eta=\eps/\alpha$ yields $k\le \exp(C/\eta)=\exp((\alpha C)/\eps)$.
\end{proof}

We now consider lower bounds. Similar as the integer power case, a direct application of Proposition~\ref{prop:key} would not provide a sharp bound.

\begin{corollary}
\label{cor:lower alpha}
Fix $\alpha>1$ and let $f_\alpha(x,y)=(\langle x,y\rangle)_+^\alpha$.
Then for every integer $m\ge 1$, $f_\alpha$ is not $(2^m,m^{-\alpha})$-stable. Equivalently, for every $\eps\in(0,1)$,
setting $m=  \eps^{-1/\alpha} $ yields a $(k,\eps)$-half-graph with
\[
k=2^m \ge \exp(c\eps^{-1/\alpha})
\]
for an absolute constant $c>0$.
\end{corollary}

\begin{proof}
Apply Proposition~\ref{prop:key} and raise the forward value $1/m$ to the power $\alpha$.
\end{proof}

To overcome this issue, we need to shift the localized pattern away from $0$ by adding a common component in an orthogonal direction.
This removes the flattening issue and restores an $\exp(\Omega(1/\eps))$ lower bound, matching the upper bound
in Proposition~\ref{prop:upper alpha} up to constants.

\begin{proposition}[Shifted lower bound]
\label{prop:shift lower}
Fix $\alpha>1$.  There exists a constant $c_\alpha>0$ such that for every $\eps\in(0,1)$,
the predicate $f_\alpha(x,y)=(\langle x,y\rangle)_+^\alpha$ is not $(k,\eps)$-stable for
\[
k \le \exp\big(c_\alpha/\eps\big).
\]
\end{proposition}

\begin{proof}
Fix an integer $m\ge 1$.  Let $\HH':=\ell_2(\mathsf E)$ and vectors $x'_s,y'_s\in \HH'$ be as in
Proposition~\ref{prop:key}, so that for $s<t$,
\[
\langle x'_s,y'_t\rangle=\frac{1}{m}\qquad\text{and}\qquad \langle x'_t,y'_s\rangle=0,
\qquad
\|x'_s\|=1,\|y'_s\|\le 1.
\]

Let $e_0$ be a unit vector orthogonal to $\HH'$, and set $\HH:=\mathrm{span}\{e_0\}\oplus \HH'$.
Fix $\theta=\pi/4$, and define for each $s\in\Sigma_m$:
\[
x_s:=\sin(\theta) e_0+\cos(\theta)x'_s,
\qquad
y_s:=\sin(\theta) e_0+\cos(\theta)y'_s.
\]
Then $\|x_s\|^2=\sin(\theta)^2+\cos(\theta)^2\|x'_s\|^2=1$ and
$\|y_s\|^2=\sin(\theta)^2+\cos(\theta)^2\|y'_s\|^2\le 1$, hence $x_s,y_s\in\U\subseteq\HH$.

For $s<t$, orthogonality of $e_0$ and $\HH'$ gives
\begin{align*}
&\langle x_s,y_t\rangle
=\sin(\theta)^2+\cos(\theta)^2\langle x'_s,y'_t\rangle
=\sin(\theta)^2+\frac{\cos(\theta)^2}{m},\\
&\langle x_t,y_s\rangle
=\sin(\theta)^2+\cos(\theta)^2\langle x'_t,y'_s\rangle
=\sin(\theta)^2.
\end{align*}
Since $\alpha>1$, the function $t\mapsto t^\alpha$ is convex and increasing on $[0,1]$, so by the mean value theorem
and monotonicity of the derivative,
\[
\big(\sin(\theta)^2+\tfrac{\cos(\theta)^2}{m}\big)^\alpha-(\sin(\theta)^2)^\alpha
\ge
\alpha(\sin(\theta)^2)^{\alpha-1}\cdot \frac{\cos(\theta)^2}{m}
=: \frac{c_{\alpha}}{m},
\]
where $c_{\alpha}=\alpha/2^{\alpha}>0$ depends only on $\alpha$.
Thus for $s<t$ we obtain
\[
f_\alpha(x_s,y_t)-f_\alpha(x_t,y_s)\ge\frac{c_{\alpha}}{m}.
\]
Enumerating $\Sigma_m$ in lexicographic order yields a $(2^m,c_{\alpha}/m)$-half-graph for $f_\alpha$.
Now given $\eps\in(0,1)$, choose $m=  c_{\alpha}/\eps $ so that $c_{\alpha}/m\ge\eps$.
Then
\[
k=2^m \ge \exp\big((\log 2)c_{\alpha}/\eps + O(1)\big),
\]
as desired.
\end{proof}

Using the same proof by replacing $\alpha$ by $d$ we obtain the sharp lower bound for integer power connective. 

\begin{corollary}[Exponential rate for integer powers]
Fix $d\in\mathbb Z^{>0}$.  Then $\langle x,y\rangle^d$ on $\U\times\U$ is not $(k,\eps)$-stable for $k\leq  \exp(c_d/\eps))$ where $c_d>0$ depending only on $d$. 
\end{corollary}
%%%%%%%%%%%%%%%%%%%%

We now record a general consequence of the tensor-power argument via polynomial approximations. It shows that no monotonicity assumption is needed for an upper bound, provided the bound is weaker in some cases as $f_\beta(x,y)$. 

Let $g\in C([-1,1])$. For $\eta>0$, define
\[
\mathcal A_g(\eta)
:=
\inf
\left\{
    \sum_{d=1}^D |a_d| :
    p(t)=a_0+\sum_{d=1}^D a_d t^d,\ 
    \|g-p\|_{L^\infty([-1,1])}\leq \eta
\right\}.
\]
By the Weierstrass approximation theorem, $\mathcal A_g(\eta)<\infty$
for every $\eta>0$.

\begin{proposition}[Continuous connectives]
Let $g\in C([-1,1])$, and define
$
f_g(x,y):=g(\langle x,y\rangle)
$
on $\U\times \U$. Then, for every $\varepsilon\in(0,1)$,  $f_g$ is $(k,\varepsilon)$-stable for all
\[
k>
\exp\left(
    \frac{2\pi\,\mathcal A_g(\varepsilon/4)}{\varepsilon}
\right).
\]
\end{proposition}

\begin{proof}
Choose a polynomial
$
p(t)=a_0+\sum_{d=1}^D a_d t^d
$
such that
\[
\|g-p\|_{L^\infty([-1,1])}\leq \eps/4. 
\]
Hence
$A:=\sum_{d=1}^D |a_d| \leq \mathcal A_g(\varepsilon/4)+o(1).
$
If $A=0$, then $p$ is constant. In that case $g$ is within
$\varepsilon/4$ of a constant, and so $f_g$ cannot have an
$(k,\varepsilon)$-half-graph for any $k$. Thus we may assume
$A>0$.

Suppose that $f_g$ has a $(k,\varepsilon)$-half-graph, witnessed by $x_1,\ldots,x_k,y_1,\ldots,y_k\in \U$. That is for every $i<j$,
$
g(\langle x_i,y_j\rangle)-g(\langle x_j,y_i\rangle)\geq \varepsilon.
$
Since $p$ approximates $g$ within $\varepsilon/4$, it follows that
\[
p(\langle x_i,y_j\rangle)-p(\langle x_j,y_i\rangle)
\geq
\varepsilon/2.
\]

Now form the Hilbert space
\[
\mathcal{G}=\bigoplus_{\substack{1\leq d\leq D\\ a_d\neq0}} \HH^{\otimes d}.
\]
For $x,y\in \U$, define
\[
X(x)=
\bigoplus_{a_d\neq0} |a_d|^{1/2}x^{\otimes d},\qquad Y(y)=
\bigoplus_{a_d\neq0} \operatorname{sgn}(a_d)|a_d|^{1/2}y^{\otimes d}.
\]
Under this notation we have
\[
\langle X(x),Y(y)\rangle_\mathcal{G}
=
\sum_{d=1}^D a_d\langle x,y\rangle^d
=
p(\langle x,y\rangle)-a_0.
\]
Moreover, as $\|X(x)\|^2\leq A,$ and $\|Y(y)\|^2\leq A.$ Therefore, for every $i<j$,
\[
\langle X(x_i),Y(y_j)\rangle\mathcal{G} - \langle X(x_j),Y(y_i)\rangle\mathcal{G} \geq \varepsilon/2.
\]
After normalizing, we obtain a $(k,\varepsilon/(2A))$-half-graph for the inner product on the unit ball of $\mathcal{G}$. By Theorem~\ref{thm:upper}, we have $k
\leq
\exp(\frac{2\pi A}{\varepsilon}).$
Taking the infimum over admissible polynomial approximants $p$, and then letting the $o(1)$ error tend to zero, gives the claim.
\end{proof}

\appendix
 \section{The continuous VC-dimension}\label{Sec: VC}

Let us consider the following definition of the continuous version of $VC$-dimension. The definition is the same as the one defined earlier by Talagrand~\cite{Talagrand87,Talagrand96}, and by Chernikov and Towsner when $k=1$ \cite{ChernikovTowsner}. 

\begin{definition}[VC-graph at margin $\eps$]
\label{def:vc graph}
Let $f:X\times Y\to\RR$ be a bounded function and fix $\eps>0$.
We say that $f$ contains a \emph{$(d,\eps)$-VC graph} if there exist points
$x_1,\dots,x_d\in X$, a threshold $s\in\RR$, and points $\{y_S: S\subseteq[d]\}\subseteq Y$
such that for every $S\subseteq[d]$ and every $i\in[d]$,
\[
\begin{cases}
f(x_i,y_S) \ge s & \text{if } i\in S,\\[2pt]
f(x_i,y_S) \le s-\eps & \text{if } i\notin S.
\end{cases}
\]
We write $\VC_\eps(f)$ for the largest $d$ such that $f$ contains a $(d,\eps)$-VC graph. 
\end{definition}

We record below the following proposition, determining the exact VC-dimension of the inner product over the unit ball in Hilbert spaces. The proof are relatively straightforward and different from the proofs for the stability of inner product. 

\begin{proposition}[Continuous VC-dimension of the Hilbert inner product]
\label{prop:vc inner product}
Let $\HH$ be a real  Hilbert space, let $\U$ be its unit ball, and let $f(x,y)=\Re\langle x,y\rangle$ on $\U\times\U$. Then for every $\eps\in(0,1)$,
\[
\VC_\eps(f)= \min\Big\{\dim(\HH),   \frac{4}{\eps^{2}} \Big\}.
\]
\end{proposition}

\begin{proof}
Let us first consider the lower bound construction. 
Let $d\le 4/\eps^{2}$ and choose orthonormal vectors $e_1,\dots,e_d\in\HH$. Set $x_i=e_i\in\U$ and choose the threshold $s=\eps/2$. For each $S\subseteq[d]$, define a sign vector $s^{(S)}\in\{\pm1\}^d$ by
\[
s^{(S)}_i=\begin{cases}
+1 & i\in S,\\
-1 & i\notin S,
\end{cases}
\qquad\text{and set}\qquad
y_S=\frac{\eps}{2}\sum_{i=1}^d s^{(S)}_ie_i.
\]
Then $\|y_S\|=\eps d^{1/2}/2\le 1$, so $y_S\in\U$, and for each $i$,
\[
f(x_i,y_S)=\langle e_i,y_S\rangle=\frac{\eps s^{(S)}_i}{2}.
\]
Hence $f(x_i,y_S)=\eps/2$ if $i\in S$, while $f(x_i,y_S)=-\eps/2$ if $i\notin S$.
Thus $f$ contains a $(d,\eps)$-VC graph, so $\VC_\eps(f)\ge   4/\eps^{2} $. 

Also from the construction,  $\VC_\eps(f)\ge \dim(\HH)$ when $\dim(\HH)<  4/\eps^{2} $. 
\medskip

Next, let us handle the upper bound. 
Assume $f$ contains a $(d,\eps)$-VC graph witnessed by $x_1,\dots,x_d\in\U$, a threshold $s\in\RR$, and points $\{y_S\}_{S\subseteq[d]}\subseteq\U$.
For each sign vector $\sigma\in\{\pm1\}^d$, let $S(\sigma)=\{i:\sigma_i=1\}$ and write $y_\sigma=y_{S(\sigma)}$.
Then for every $i$,
\[
\sigma_i\big(f(x_i,y_\sigma)-s+\eps/2\big) \ge \eps/2.
\]
where $\sigma_i\in\{\pm1\}$ is the value of $\sigma$ at the $i$-th coordinate. 
Summing over $i$ gives
\[
\sum_{i=1}^d \sigma_i f(x_i,y_\sigma) - s\sum_{i=1}^d\sigma_i + \frac{\eps}{2}\sum_{i=1}^d\sigma_i
 \ge \frac{d\eps}{2}.
\]
Write $m(\sigma)=\sum_{i=1}^d\sigma_i$.  Since $f(x_i,y_\sigma)=\langle x_i,y_\sigma\rangle$ and $\|y_\sigma\|\le 1$,
Cauchy--Schwarz yields
\[
\sum_{i=1}^d \sigma_i f(x_i,y_\sigma)
=\Big\langle \sum_{i=1}^d \sigma_i x_i, y_\sigma\Big\rangle
\le \Big\|\sum_{i=1}^d \sigma_i x_i\Big\|.
\]
Thus for every $\sigma$,
\begin{equation}\label{eq:key-vc-upper}
\Big\|\sum_{i=1}^d \sigma_i x_i\Big\|
 \ge \frac{d\eps}{2} + (s-\eps/2) m(\sigma).
\end{equation}

Now we apply \eqref{eq:key-vc-upper} to both $\sigma$ and $-\sigma$ and add the two inequalities.
Since $m(-\sigma)=-m(\sigma)$, the $m(\sigma)$-terms cancel each other and we obtain
\[
\Big\|\sum_{i=1}^d \sigma_i x_i\Big\|+\Big\|\sum_{i=1}^d (-\sigma_i) x_i\Big\|
\ge d\eps.
\]
Observe that the two norms on the left are equal, so for every  $\sigma\in\{\pm1\}^d$, 
\[
\Big\|\sum_{i=1}^d \sigma_i x_i\Big\| \ge \frac{d\eps}{2}.
\]
Squaring and averaging over uniform random $\sigma$ gives
\[
\frac{d^2\eps^2}{4}
 \le \mathbb{E} \Big\|\sum_{i=1}^d \sigma_i x_i\Big\|^2
 = \sum_{i,j=1}^d\mathbb{E}[\sigma_i\sigma_j] \langle x_i,x_j\rangle= \sum_{i=1}^d \|x_i\|^2
 \le d,
\]
where the equality uses $\mathbb{E}[\sigma_i\sigma_j]=\mathbb{E}[\sigma_i]\mathbb{E}[\sigma_j]=0$ for $i\neq j$.
Hence $d\le 4\eps^{-2}$.

It remains to consider the case when $4\eps^{-2}<\dim(\HH)$, and this is done by Lemma~\ref{lem: bounded dim} below.  
\end{proof}

\begin{lemma}
\label{lem: bounded dim}
Let $\HH$ be a real Hilbert space of dimension $n$, let $\U$ be its unit ball, and let
$
f(x,y)=\langle x,y\rangle
$ on $\U\times \U$. 
Then for every $\eps\in(0,1)$,
\[
\VC_\eps(f)\le n.
\]
\end{lemma}

\begin{proof}
Suppose toward contradiction that $f$ contains a $(d,\eps)$-VC graph with $d>n$.
Thus there exist points
\[
x_1,\dots,x_d\in \U,
\qquad
\{y_S: S\subseteq[d]\}\subseteq \U,
\]
and a threshold $s\in\RR$ witness the graph. 

Let
$
V=\operatorname{span}\{x_1,\dots,x_d\}\subseteq \HH.
$
Then $\dim V\le n$. For each $y\in \HH$, only its orthogonal projection to $V$ matters for the values
$\langle x_i,y\rangle$, so we may replace every $y_S$ by its projection $P_V y_S$  and work entirely inside $V$.

For each $i\in[d]$, consider the affine hyperplane
\[
H_i=\{y\in V:\langle x_i,y\rangle = s -\eps/2\}.
\]
Hence the sign pattern
$
({\mathbbm 1}_{\{\langle x_1,y\rangle>s-\eps/2\}},\dots,{\mathbbm 1}_{\{\langle x_d,y\rangle>s-\eps/2\}})
$
is constant on each cell of the hyperplane arrangement $\{H_1,\dots,H_d\}$ in $V$.
Since the family $\{y_S: S\subseteq[d]\}$ realizes all $2^d$ subsets of $[d]$, this arrangement must have at least
$2^d$ cells.

On the other hand, Zaslavsky's theorem~\cite{Zaslavsky} claims that the number of cells cut out by $d$ affine hyperplanes
in an $r$-dimensional real vector space is at most
$
\sum_{j=0}^{r} \binom{d}{j},
$
where here $r=\dim V\le n$.
Since $d>r$, we have
$
\sum_{j=0}^{r}\binom{d}{j}<2^d.
$
Thus the arrangement cannot realize all $2^d$ subsets, a contradiction.
Therefore $d\le n$, and hence $\VC_\eps(f)\le n$.
\end{proof}

Let us remark that if $\HH$ is a complex Hilbert space, by replacing $\langle x,y\rangle$ by $\Re\langle x,y\rangle$, the same proof in Proposition~\ref{prop:vc inner product} still works, but we would only obtain $2\dim(\HH)$ from Lemma~\ref{lem: bounded dim}, so the statement for complex Hilbert space in Proposition~\ref{prop:vc inner product} should be change to $\min\{2\dim(\HH),   4/\eps^2 \}$.

\section{A matching lower bound construction}

For an integer $n\geq 2$, let $J_n=(J_{ij})_{i,j=1}^n$ be the real skew symmetric matrix with $J_{ij}=1$ when $i<j$ and with all diagonals equal to $0$. Define the trace norm (also known as Schatten $1$-norm) of a matrix the sum of its singular value, then we have the following observation. \
\begin{lemma}
    We have $\|J_n\|_{S_1}=\frac{1}{n}\sum_{r=1}^n
        |\cot(\frac{(2r-1)\pi}{2n})|$. 
\end{lemma}
\begin{proof}
 We compute the singular values of $J_n$. Write
  \[
    \zeta_r:=\exp\!\left(\frac{(2r-1)\pi\mathrm{i}}{n}\right),
    \qquad
    v_r:=\frac{1}{n^{1/2}}(1,\zeta_r,\ldots,\zeta_r^{n-1})^{\mathsf T},
  \]
  one has $\zeta_r^n=-1$, and a geometric sum calculation gives
  \[
    J_nv_r
      =\frac{1+\zeta_r}{1-\zeta_r}v_r
      =\mathrm{i}\cot\!\left(\frac{(2r-1)\pi}{2n}\right)v_r.
  \]
  It follows that the singular values of
  $J_n$ are
  \[
    \left|\cot\!\left(\frac{(2r-1)\pi}{2n}\right)\right|
  \]
  for $1\leq r\leq n$. 
\end{proof}

Now we write
\[
  c_n=\frac{1}{n}\lVert J_n\rVert_{S_1}
      =\frac{1}{n}\sum_{r=1}^n
        \left|\cot\!\left(\frac{(2r-1)\pi}{2n}\right)\right|.
\]
Let us first estimate the size of $c_n$ by the following computational lemma. 

\begin{lemma}\label{lem: estimate c_n}
  For every integer $n\geq 2$, we have
  \[
    c_n\leq \frac{2}{\pi}(\log n+1).
  \]
\end{lemma}
\begin{proof}
Write $h= n/2$ and define $S_h=\sum_{j=1}^h(j-1/2)^{-1}$. Observe that for $0<x<\pi/2$, $\cot x \leq 1/x - x/3$. 
Applying this term by term to the exact formula for $c_n$ yields
  \begin{equation}\label{eq: c_n}
    \frac{\pi}{2}c_n
      \leq S_h -\frac{\pi^2h^2}{6n^2}.                        
  \end{equation}
  The standard harmonic sum estimates yield
  \[
    S_h\leq\log(4h)+\gamma+\frac{1}{2h(2h+1)}.
  \]
  If $n=2h$, subtracting $\log n$ from the right hand side of \eqref{eq: c_n} is strictly less than $1$. Indeed, for $h\geq2$, it is at most $\log2+\gamma+1/20-\pi^2/24<1$, and $h=1$ is immediate. If $n=2h+1$, the corresponding quantity is at most
  \[
    \log\frac{4h}{2h+1}+\gamma+\frac{1}{2h(2h+1)}
      -\frac{\pi^2h^2}{6(2h+1)^2}<1.
  \]
  Indeed, $h=1,2$ are immediate, while for $h\geq3$ the last display is at most $\log2+\gamma+1/42-3\pi^2/98<1$. This proves the claim.
\end{proof}

We are now able to prove our main theorem of the section, which answers Question~\ref{question}.

\begin{theorem}
      Let $0<\varepsilon<1$ and write
      \[
    N=\exp\!\left(\frac{\pi}{\varepsilon}-1\right)
  \]
  The inner product on the unit ball of any Hilbert space $\HH$ of dimension at least $N$ is not $(k,\varepsilon)$-stable for every $k\leq N$.
\end{theorem}

\begin{proof}
Write $\Sigma = \mathrm{diag}(\sigma_1,\dots,\sigma_n)$ where $\sigma_i$ are singular values of $J_n$. 
Consider a real singular value decomposition
\[
    J_n=A\Sigma B^{\mathsf T}
\]
and define $X=c_n^{-1/2}A\Sigma^{1/2}$ and $Y=c_n^{-1/2}B\Sigma^{1/2}$. Let $x_i$ and $y_i$ be the $i$-th rows of $X$ and $Y$, regarded as vectors in $\mathbb R^n\subseteq \HH$. Clearly we have
\[
    \lVert x_i\rVert^2
      =\frac{(A\Sigma A^{\mathsf T})_{ii}}{c_n}=1,
    \qquad
    \lVert y_i\rVert^2
      =\frac{(B\Sigma B^{\mathsf T})_{ii}}{c_n}=1.
\]
On the other hand, 
\[
    \langle x_i,y_j\rangle=(XY^{\mathsf T})_{ij}
      =\frac{J_{ij}}{c_n}.
\]
Thus, whenever $i<j$,
\[
    \langle x_i,y_j\rangle-\langle x_j,y_i\rangle
      =\frac{J_{ij}-J_{ji}}{c_n}=\frac{2}{c_n}.
\]
By Lemma~\ref{lem: estimate c_n}, we conclude that
\[
\langle x_i,y_j\rangle-\langle x_j,y_i\rangle\geq \frac{\pi}{\log n +1}
\]
and the theorem follows. 
\end{proof}

Let us remark that the value $c_n$ is the $\gamma_2$-norm of $J_n$, which is also known as the Hilbert space factorization norm and is equivalent to the cut norm defined in \eqref{eq: cut}. 

\bibliographystyle{amsalpha}
\bibliography{ref}

\providecommand{\bysame}{\leavevmode\hbox to3em{\hrulefill}\thinspace}
\providecommand{\MR}{\relax\ifhmode\unskip\space\fi MR }
\providecommand{\MRhref}[2]{%
  \href{http://www.ams.org/mathscinet-getitem?mr=#1}{#2}
}
\providecommand{\href}[2]{#2}
\begin{thebibliography}{BYBHU08}

\bibitem[BMMN13]{BMMN}
Mark Braverman, Konstantin Makarychev, Yury Makarychev, and Assaf Naor,
  \emph{The {G}rothendieck constant is strictly smaller than {K}rivine's
  bound}, Forum Math. Pi \textbf{1} (2013), e4, 42. \MR{3141414}

\bibitem[BY14]{BenStab}
Ita\"i Ben~Yaacov, \emph{Model theoretic stability and definability of types,
  after {A}. {G}rothendieck}, Bull. Symb. Log. \textbf{20} (2014), no.~4,
  491--496. \MR{3294276}

\bibitem[BYBHU08]{Continuous_Logic}
Ita\"i{} Ben~Yaacov, Alexander Berenstein, C.~Ward Henson, and Alexander
  Usvyatsov, \emph{Model theory for metric structures}, Model theory with
  applications to algebra and analysis. {V}ol. 2, London Math. Soc. Lecture
  Note Ser., vol. 350, Cambridge Univ. Press, Cambridge, 2008, pp.~315--427.
  \MR{2436146}

\bibitem[CCP24]{CCP24}
Nicolas Chavarria, Gabriel Conant, and Anand Pillay, \emph{Continuous stable
  regularity}, J. Lond. Math. Soc. (2) \textbf{109} (2024), no.~1, Paper No.
  e12822, 36. \MR{4680211}

\bibitem[Con21]{Conant_Stab}
Gabriel Conant, \emph{Stability in a group}, Groups Geom. Dyn. \textbf{15}
  (2021), no.~4, 1297--1330. \MR{4349660}

\bibitem[CT20]{ChernikovTowsner}
Artem Chernikov and Henry Towsner, \emph{Hypergraph regularity and higher arity
  vc-dimension}, arXiv:2010.00726 (2020).

\bibitem[FHMZ05]{FHMZ}
M.~Fabian, P.~H\'ajek, V.~Montesinos, and V.~Zizler, \emph{A quantitative
  version of {K}rein's theorem}, Rev. Mat. Iberoamericana \textbf{21} (2005),
  no.~1, 237--248. \MR{2155020}

\bibitem[Gra06]{Granero}
Antonio~S. Granero, \emph{An extension of the {K}rein-\v smulian theorem}, Rev.
  Mat. Iberoam. \textbf{22} (2006), no.~1, 93--110. \MR{2267314}

\bibitem[Gro52]{Grothendieck}
A.~Grothendieck, \emph{Crit\`eres de compacit\'e{} dans les espaces
  fonctionnels g\'en\'eraux}, Amer. J. Math. \textbf{74} (1952), 168--186.
  \MR{47313}

\bibitem[Gro53]{GrothendieckIneq}
\bysame, \emph{R\'esum\'e{} de la th\'eorie m\'etrique des produits tensoriels
  topologiques}, Bol. Soc. Mat. S\~ao Paulo \textbf{8} (1953), 1--79.
  \MR{94682}

\bibitem[Hru12]{Hrushovski}
Ehud Hrushovski, \emph{Stable group theory and approximate subgroups}, J. Amer.
  Math. Soc. \textbf{25} (2012), no.~1, 189--243. \MR{2833482}

\bibitem[KKS13]{QuantitativeBanach}
Miroslav Ka{\v c}ena, Ond{\v r}ej F.~K. Kalenda, and Ji{\v r}{\'i} Spurn{\'y},
  \emph{Quantitative {D}unford-{P}ettis property}, Adv. Math. \textbf{234}
  (2013), 488--527. \MR{3003935}

\bibitem[KM81]{KrivineMaurey}
J.-L. Krivine and B.~Maurey, \emph{Espaces de {B}anach stables}, Israel J.
  Math. \textbf{39} (1981), no.~4, 273--295. \MR{636897}

\bibitem[Kri77]{Krivine_rounding}
Jean-Louis Krivine, \emph{Sur la constante de {G}rothendieck}, C. R. Acad. Sci.
  Paris S\'er. A-B \textbf{284} (1977), no.~8, A445--A446. \MR{428414}

\bibitem[Pil16]{Pillay_Groth}
Anand Pillay, \emph{Generic stability and {G}rothendieck}, South Amer. J. Log.
  \textbf{2} (2016), no.~2, 437--442. \MR{3671044}

\bibitem[PS13]{Pillay_Starchenko}
Anand Pillay and Sergei Starchenko, \emph{Remarks on {T}ao's algebraic
  regularity lemma}, unpublished (2013).

\bibitem[Tal87]{Talagrand87}
Michel Talagrand, \emph{The {G}livenko-{C}antelli problem}, Ann. Probab.
  \textbf{15} (1987), no.~3, 837--870. \MR{893902}

\bibitem[Tal96]{Talagrand96}
\bysame, \emph{The {G}livenko-{C}antelli problem, ten years later}, J. Theoret.
  Probab. \textbf{9} (1996), no.~2, 371--384. \MR{1385403}

\bibitem[Tao13]{TaoBlog}
Terence Tao, \emph{A spectral theory proof of the algebraic regularity lemma},
  blog post available at
  \url{https://terrytao.wordpress.com/2013/10/29/a-spectral-theory-proof-of-the-algebraic-regularity-lemma/}
  (2013).

\bibitem[Tao15]{Tao}
Terence Tao, \emph{Expanding polynomials over finite fields of large
  characteristic, and a regularity lemma for definable sets}, Contrib. Discrete
  Math. \textbf{10} (2015), no.~1, 22--98. \MR{3386249}

\bibitem[Zas75]{Zaslavsky}
Thomas Zaslavsky, \emph{Facing up to arrangements: face-count formulas for
  partitions of space by hyperplanes}, Mem. Amer. Math. Soc. \textbf{1} (1975),
  vii+102. \MR{357135}

\end{thebibliography}

\end{document}